\documentclass[11pt]{amsart}

\usepackage[margin=1.25in]{geometry}

\usepackage{graphicx}

\usepackage{amsmath, amssymb, amsthm, graphicx, enumerate, tikz, float, bbm, nicefrac}
\usepackage{mathtools, comment}
\usepackage[colorlinks]{hyperref}
\hypersetup{citecolor=blue}
\usetikzlibrary{matrix,arrows,decorations.pathmorphing}

\newtheorem{theorem}{Theorem}[section]

\newtheorem{corollary}[theorem]{Corollary}
\newtheorem{proposition}[theorem]{Proposition}

\theoremstyle{definition}
\newtheorem{definition}[theorem]{Definition}
\newtheorem{example}[theorem]{Example}

\newtheorem*{solution*}{Solution}

\theoremstyle{remark}
\newtheorem{remark}[theorem]{Remark}
\newtheorem*{pf*}{Pf}
\newtheorem*{pfbase*}{Proof of Base Case}
\newtheorem*{pfstep*}{Proof of Inductive Step}

\numberwithin{equation}{section}

\DeclareMathOperator{\Hom}{Hom}
\DeclareMathOperator{\Hg}{H}
\DeclareMathOperator{\Der}{Der}
\DeclareMathOperator{\Inn}{Inn}
\DeclareMathOperator{\Ig}{Im}
\DeclareMathOperator{\Ker}{Ker}
\DeclareMathOperator{\id}{id}

\begin{document}

\allowdisplaybreaks

\title[Deformation theories controlled by Hochschild cohomologies]{Deformation theories controlled by Hochschild cohomologies}

\author{Samuel Carolus}
\address{Department of Mathematics and Statistics, Ohio Northern University, Ada, Ohio 45810}
\email{s-carolus@onu.edu}

\author{Samuel A. Hokamp}
\address{Department of Mathematics, St. Norbert College, De Pere, Wisconsin 54115}
\email{samuel.hokamp@snc.edu}

\author{Jacob Laubacher}
\address{Department of Mathematics, St. Norbert College, De Pere, Wisconsin 54115}
\email{jacob.laubacher@snc.edu}

\subjclass[2010]{Primary 16S80; Secondary 16E40}

\date{\today}

\keywords{Deformations of algebras, higher order Hochschild cohomology, tertiary Hochschild cohomology}

\begin{abstract}
We explore how the higher order Hochschild cohomology controls a deformation theory when the simplicial set models the 3-sphere. Besides generalizing to the $d$-sphere for any $d\geq1$, we also investigate a deformation theory corresponding to the tertiary Hochschild cohomology, which naturally reduces to those studied for the secondary and usual Hochschild cohomologies under certain conditions.
\end{abstract}

\maketitle

\section{Introduction}

Higher order Hochschild (co)homology was implicitly defined by Anderson in \cite{A}, and was given an explicit description in \cite{P}. The case for when the simplicial set models the $d$-sphere was investigated more extensively in \cite{G}. A deformation theory for the algebra $A[[t]]$ controlled by the higher order Hochschild cohomology over the 2-sphere was studied in \cite{CS}. One of the goals of this paper is to generalize their argument.

In Section \ref{SectionHigh} we use the simplicial structure for the 3-sphere presented in \cite{CL}, and also use their natural extension when considering the $d$-sphere for any $d\geq1$. We show how the higher order Hochschild cohomology over the $d$-sphere controls a deformation theory, giving special attention to the case when $d=3$.

When the simplicial set models $S^1$, it is well known that one recovers the usual Hochschild cohomology, which was introduced in 1945 in \cite{H}. Almost twenty years later in \cite{Ger}, Gerstenhaber used this Hochschild cohomology, denoted $\Hg^*(A,A)$, to describe deformations of the algebra $A[[t]]$.  That is, for a multiplication law $m_t:A[[t]]\otimes A[[t]]\longrightarrow A[[t]]$ determined by $m_t(a\otimes b)=ab+c_1(a\otimes b)t+c_2(a\otimes b)t^2+\cdots$ with $\mathbbm{k}$-linear maps $c_i:A\otimes A\longrightarrow A$, one sees that $m_t$ is associative $mod~t^2$ if and only if $c_1$ is a $2$-cocycle. As is classical, the class of $c_1$ is determined by the isomorphism class of $m_t$. Finally, assuming associativity $mod~t^{n+1}$, the obstruction for associativity $mod~t^{n+2}$ is an element in $\Hg^3(A,A)$. 

In 2016, Staic showed in \cite{S} that when one wants to study deformations of $A[[t]]$ that have a nontrivial $B$-algebra structure, one can use the secondary Hochschild cohomology. This cohomology theory has the property that when one takes $B=\mathbbm{k}$, one recovers the usual Hochschild cohomology.

In Section \ref{SectionTert} we study deformations of $A[[t]]$ that have nontrivial $B$-algebra and $C$-algebra structures. This is done using the tertiary Hochschild cohomology, which was introduced in \cite{CL}. This tertiary Hochschild cohomology depends on a morphism of commutative $\mathbbm{k}$-algebras $\theta:C\longrightarrow B$.  This morphism induces a $B$-algebra and $C$-algebra structure on $A$ by way of the morphisms $\varepsilon:B\longrightarrow A$ and $\varepsilon\circ\theta:C\longrightarrow A$, respectively. We show that this is equivalent to having a family of products satisfying a generalized associativity condition. In particular, when one takes $C=\mathbbm{k}$, one recovers exactly the result in \cite{S}. Also, as a natural extension, we discuss deformations of $A[[t]]$ with $n$ nontrivial algebra structures for any $n\geq1$.

\section{Preliminaries} 

Fix $\mathbbm{k}$ to be a field and denote $\otimes:=\otimes_\mathbbm{k}$ and $\Hom(-,-):=\Hom_\mathbbm{k}(-,-)$. Furthermore, we set $A$ to be an associative $\mathbbm{k}$-algebra with multiplicative unit.

For $d\geq1$, we begin by recalling the chain complex associated to the higher order Hochschild cohomology of the commutative $\mathbbm{k}$-algebra $A$ with coefficients in the $A$-symmetric $A$-bimodule $M$ over the $d$-sphere $S^d$. We denote the complex
\begin{equation}\label{SdComplex}
0\to M\xrightarrow{\delta_0}\ldots\xrightarrow{\delta_{d-2}}M\xrightarrow{\delta_{d-1}}\Hom(A,M)\xrightarrow{\delta_d}\Hom(A^{\otimes d+1},M)\xrightarrow{\delta_{d+1}}\ldots
\end{equation}
by $\mathbf{C}_{S^d}^\bullet(A,M)$. It will be of particular interest when one takes $M=A$. Moreover, we focus on the map $\delta_d:\Hom(A,M)\longrightarrow\Hom(A^{\otimes d+1},M)$. For any $\mathbbm{k}$-linear map $f:A\longrightarrow M$, we have that
\begin{equation}\label{SdMap}
\begin{aligned}
\delta_d(f)(a_1\otimes\cdots\otimes a_{d+1})&=a_1\cdots a_df(a_{d+1})\\
&\hspace{.25in}+\sum_{i=1}^{d}(-1)^ia_1\cdots a_{d-i}f(a_{d+1-i}a_{d+2-i})a_{d+3-i}\cdots a_{d+1}\\
&\hspace{.25in}+(-1)^{d+1}f(a_1)a_2\cdots a_{d+1}.\\
\end{aligned}
\end{equation}

\begin{definition}\label{SdDefn}(\cite{A},\cite{P})
The cohomology of the complex $\mathbf{C}_{S^d}^\bullet(A,M)$ is called the \textbf{higher order Hochschild cohomology} of $A$ with coefficients in $M$ over the $d$-sphere, which is denoted by $\Hg_{S^d}^*(A,M)$.
\end{definition}

We note that when taking $M=A$, then for any $d\geq1$, the maps in low dimension are straightforward. Indeed for $0\leq n\leq d-1$, we have that
\begin{equation}\label{evenodd}
\delta_n=\begin{cases}
0 & \text{if $n$ is even,}\\
\id & \text{if $n$ is odd.}
\end{cases}
\end{equation}

As consequence of \eqref{evenodd}:

\begin{example}
When $M=A$, we have the following for any $d\geq1$:
\begin{enumerate}[(i)]
\item $\Hg_{S^d}^0(A,A)=A$,
\item $\Hg_{S^d}^n(A,A)=0$ for all $1\leq n\leq d-1$,
\item $\Hg_{S^d}^d(A,A)\cong\Ker(\delta_d)$ for when $d$ is odd, and
\item $\Hg_{S^d}^d(A,M)\cong\frac{\Ker(\delta_d)}{A}$ for when $d$ is even.
\end{enumerate}
\end{example}

Next, we recall the tertiary Hochschild cohomology of a $\mathbbm{k}$-algebra $A$. This algebra $A$ need not be commutative, unlike the case for the higher order Hochschild cohomology.  The tertiary Hochschild homology was introduced in \cite{CL}, and the cohomology is an easy adaptation, as they mentioned. For the purposes of this paper, it suffices to only consider the complex in low dimension.

\begin{definition}(\cite{CL})\label{Quin}
We call $(A,B,C,\varepsilon,\theta)$ a \textbf{quintuple} if
\begin{enumerate}[(i)]
    \item $A$ is a $\mathbbm{k}$-algebra,
    \item $B$ is a commutative $\mathbbm{k}$-algebra,
    \item $\varepsilon:B\longrightarrow A$ is a morphism of $\mathbbm{k}$-algebras such that $\varepsilon(B)\subseteq\mathcal{Z}(A)$,
    \item $C$ is a commutative $\mathbbm{k}$-algebra, and
    \item $\theta:C\longrightarrow B$ is a morphism of  $\mathbbm{k}$-algebras.
\end{enumerate}
\end{definition}

We next consider a quintuple $\mathcal{Q}=(A,B,C,\varepsilon,\theta)$, and we let $M$ be an $A$-bimodule which is $B$-symmetric (and therefore $C$-symmetric). We denote the complex
$$
0\to M\xrightarrow{\gimel^0}\Hom(A,M)\xrightarrow{\gimel^1}\Hom(A^{\otimes2}\otimes B\otimes C,M)\xrightarrow{\gimel^2}\Hom(A^{\otimes3}\otimes B^{\otimes3}\otimes C^{\otimes4},M)\xrightarrow{\gimel^3}\ldots
$$
by $\mathbf{C}^\bullet(\mathcal{Q};M)$. Again, it will be of particular interest when one takes $M=A$. First, however, we describe these maps in low dimension. As noted in \cite{CL}, one can arrange these elements to form a tetrahedron. Since working with an element expressed in three dimensions is laborious, we instead follow the arrangement in \cite{CL} and consider a two-dimensional sliced representation. For ease of notation, we will consider elements $a,b,c\in A$, $\alpha,\beta,\gamma\in B$, and $x,y,z,w\in C$:
\begin{align*}
\gimel^0(f)(a)&=af(1)-f(1)a,\\ 
\gimel^1(f)\Big(\begin{pmatrix}a\end{pmatrix}\otimes\begin{pmatrix}x & \alpha\\ & b\\\end{pmatrix}\Big)&=a\varepsilon(\alpha\theta(x))f(b)-f(ab\varepsilon(\alpha\theta(x)))+f(a)b\varepsilon(\alpha\theta(x)),
\end{align*}
and
\begin{align*}
\gimel^2(f)&\Big(\begin{pmatrix}a\\\end{pmatrix}\otimes\begin{pmatrix} x & \alpha\\ & b\\\end{pmatrix}\otimes\begin{pmatrix} y & z & \beta\\ & w & \gamma\\ & & c\end{pmatrix}\Big)\\
&=a\varepsilon(\alpha\beta\theta(xyz))f\Big(\begin{pmatrix} b\\\end{pmatrix}\otimes\begin{pmatrix} w & \gamma\\ & c\end{pmatrix}\Big)-f\Big(\begin{pmatrix} ab\varepsilon(\alpha\theta(x))\\\end{pmatrix}\otimes\begin{pmatrix} yzw & \beta\gamma\\ & c\end{pmatrix}\Big)\\
&\hspace{.25in}+f\Big(\begin{pmatrix} a\\\end{pmatrix}\otimes\begin{pmatrix} xy & \alpha\beta\theta(z)\\ & bc\varepsilon(\gamma\theta(w))\end{pmatrix}\Big)-f\Big(\begin{pmatrix} a\\\end{pmatrix}\otimes\begin{pmatrix} x & \alpha\\ & b\end{pmatrix}\Big)c\varepsilon(\beta\gamma\theta(yzw)).
\end{align*}

\begin{definition}(\cite{CL})
Let $\mathcal{Q}=(A,B,C,\varepsilon,\theta)$ be a quintuple. The cohomology of the complex $\mathbf{C}^\bullet(\mathcal{Q};M)$ is called the \textbf{tertiary Hochschild cohomology} of the quintuple $(A,B,C,\varepsilon,\theta)$ with coefficients in $M$, which is denoted by $\Hg^*((A,B,C,\varepsilon,\theta);M)$.
\end{definition}

\begin{example}(\cite{CL})\label{UST}
When $C=\mathbbm{k}$, one recovers the secondary Hochschild cohomology $\Hg^*((A,B,\varepsilon);M)$, introduced in \cite{S}. When $C=B=\mathbbm{k}$, one recovers the usual Hochschild cohomology $\Hg^*(A,M)$, as in \cite{H}.
\end{example}

\begin{example}
In low dimensions, one can see $\Hg^0((A,B,C,\varepsilon,\theta);M)=\Hg^0((A,B,\varepsilon);M)=\Hg^0(A,M)=[M,A]$ and $\Hg^1((A,B,C,\varepsilon,\theta);M)=\frac{\Der_{B,C}(A,M)}{\Inn_\mathbbm{k}(A,M)}$. Here $\Der_{B,C}(A,M)$ denotes the module of all derivations of $A$ in $M$ which are both $B$-linear and $C$-linear, and $\Inn_\mathbbm{k}(A,M)$ denotes the inner derivations.
\end{example}

\section{Higher order Hochschild cohomology}\label{SectionHigh}

For this section we fix $A$ to be commutative. Consider a $\mathbbm{k}[[t]]$-linear map $u:A[[t]]\longrightarrow A[[t]]$ determined by 
\begin{equation}\label{u}
u(a)=a+u_1(a)t+u_2(a)t^2+u_3(a)t^3+u_4(a)t^4+\cdots
\end{equation}
where each $u_i:A\longrightarrow A$ is $\mathbbm{k}$-linear.

We note that \eqref{u} was investigated in \cite{CS} and an associativity-like condition for three elements was shown to be controlled by $\Hg_{S^2}^*(A,A)$. Here, we focus on $\Hg_{S^3}^*(A,A)$, and then ultimately generalize to $\Hg_{S^d}^*(A,A)$ for any $d\geq1$.

\subsection{Modeling the 3-sphere}\label{3sphereSub}

We start by recalling the complex associated to $\Hg_{S^3}^*(A,A)$. We get
$$
0\to A\xrightarrow{\delta_0}A\xrightarrow{\delta_1}A\xrightarrow{\delta_2}\Hom(A,A)\xrightarrow{\delta_3}\Hom(A^{\otimes4},A)\xrightarrow{\delta_4}\ldots
$$
by taking $d=3$ in \eqref{SdComplex} with $M=A$. Just as in \eqref{SdMap}, we want to focus on the map $\delta_3:\Hom(A,A)\longrightarrow\Hom(A^{\otimes4},A)$. For any $\mathbbm{k}$-linear map $f:A\longrightarrow A$, we have that
$$
\delta_3(f)(a\otimes b\otimes c\otimes d)=abcf(d)-abf(cd)+af(bc)d-f(ab)cd+f(a)bcd.
$$
Next we consider two $\mathbbm{k}$-linear maps $f,g:A\longrightarrow A$. We define $f\circ g:A^{\otimes4}\longrightarrow A$ by
$$
(f\circ g)(a\otimes b\otimes c\otimes d)=f(ab)g(cd)-af(bc)g(d)-f(a)bcg(d)-f(a)g(bc)d.
$$
Furthermore, for three $\mathbbm{k}$-linear maps $f,g,h:A\longrightarrow A$, we define $f\star g\star h:A^{\otimes4}\longrightarrow A$ by
$$
(f\star g\star h)(a\otimes b\otimes c\otimes d)=-f(a)g(bc)h(d).
$$

Suppose we desire the map $u$ from \eqref{u} to satisfy the property
\begin{equation}\label{sthree}
u(ab)u(cd)=u(a)u(bc)u(d).
\end{equation}

\begin{proposition}
Let $u:A[[t]]\longrightarrow A[[t]]$ be defined as in \eqref{u}.
\begin{enumerate}[(i)]
\item If $u$ satisfies \eqref{sthree} $mod~t^2$, then $u_1\in Z_{S^3}^3(A,A)$.
\item If $u$ satisfies \eqref{sthree} $mod~t^{n+1}$, then $u$ can be extended so that it satisfies \eqref{sthree} $mod~t^{n+2}$ if and only if
$$\sum_{i+j=n+1}u_i\circ u_j+\sum_{i+j+k=n+1}u_i\star u_j\star u_k=0\in\Hg_{S^3}^4(A,A).$$
\end{enumerate}
\end{proposition}
\begin{proof}
First observe that satisfying \eqref{sthree} yields
$$
\big(ab+u_1(ab)t+u_2(ab)t^2+u_3(ab)t^3+\cdots\big)\big(cd+u_1(cd)t+u_2(cd)t^2+u_3(cd)t^3+\cdots\big)
$$
$$
=\big(a+u_1(a)t+u_2(a)t^2+\cdots\big)\big(bc+u_1(bc)t+u_2(bc)t^2+\cdots\big)\big(d+u_1(d)t+u_2(d)t^2+\cdots\big).
$$

For $(i)$, we notice that in order to satisfy \eqref{sthree} $mod~t^2$ we would need
$$(ab)(cd)+abu_1(cd)t+u_1(ab)cdt=a(bc)d+abcu_1(d)t+au_1(bc)dt+u_1(a)bcdt.$$
This means
$$abcu_1(d)-abu_1(cd)+au_1(bc)d-u_1(ab)cd+u_1(a)bcd=0,$$
and hence $\delta_3(u_1)(a\otimes b\otimes c\otimes d)=0$. Therefore $u_1\in\Ker(\delta_3)$ and so $u_1$ is a 3-cocycle. Thus $u_1\in Z_{S^3}^3(A,A)$.

For $(ii)$, we will show the cases in low dimension with the extension following naturally. For $n=1$, if \eqref{sthree} is satisfied $mod~t^2$, and we desire equality $mod~t^3$, then it reduces to
\begin{equation}\label{3spheret3}
\begin{gathered}
abu_2(cd)+u_2(ab)cd+u_1(ab)u_1(cd)\\
=abcu_2(d)+au_2(bc)d+u_2(a)bcd+au_1(bc)u_1(d)+u_1(a)bcu_1(d)+u_1(a)u_1(bc)d.
\end{gathered}
\end{equation}
One arranges \eqref{3spheret3} to become
\begin{equation}\label{3spheret32}
\begin{gathered}
abcu_2(d)-abu_2(cd)+au_2(bc)d-u_2(ab)cd+u_2(a)bcd\\
=u_1(ab)u_1(cd)-au_1(bc)u_1(d)-u_1(a)bcu_1(d)-u_1(a)u_1(bc)d.
\end{gathered}
\end{equation}
Writing \eqref{3spheret32} in a nice way yields $\delta_3(u_2)(a\otimes b\otimes c\otimes d)=(u_1\circ u_1)(a\otimes b\otimes c\otimes d)$. Thus, $u_1\circ u_1\in\Ig(\delta_3)$, and therefore $u_1\circ u_1=0\in\Hg_{S^3}^4(A,A)$, which was what we wanted.

Notice how the latter sum is suppressed in the case $n=1$. For $n=2$, if we suppose \eqref{sthree} is satisfied $mod~t^3$, and we desire equality $mod~t^4$, then it reduces to
\begin{equation}\label{3spheret4}
\begin{gathered}
abu_3(cd)+u_3(ab)cd+u_1(ab)u_2(cd)+u_2(ab)u_1(cd)\\
=abcu_3(d)+au_3(bc)d+u_3(a)bcd+au_1(bc)u_2(d)+u_1(a)bcu_2(d)+u_1(a)u_2(bc)d\\
+au_2(bc)u_1(d)+u_2(a)bcu_1(d)+u_2(a)u_1(bc)d+u_1(a)u_1(bc)u_1(d).
\end{gathered}
\end{equation}
Rewriting \eqref{3spheret4} yields
\begin{equation}\label{3spheret42}
\begin{gathered}
abcu_3(d)-abu_3(cd)+au_3(bc)d-u_3(ab)cd+u_3(a)bcd\\
=u_1(ab)u_2(cd)-au_1(bc)u_2(d)-u_1(a)bcu_2(d)-u_1(a)u_2(bc)d\\
+u_2(ab)u_1(cd)-au_2(bc)u_1(d)-u_2(a)bcu_1(d)-u_2(a)u_1(bc)d-u_1(a)u_1(bc)u_1(d).
\end{gathered}
\end{equation}
Notice that \eqref{3spheret42} is $\delta_3(u_3)(a\otimes b\otimes c\otimes d)=(u_1\circ u_2+u_2\circ u_1+u_1\star u_1\star u_1)(a\otimes b\otimes c\otimes d)$. Thus $u_1\circ u_2+u_2\circ u_1+u_1\star u_1\star u_1\in\Ig(\delta_3)$, and therefore $u_1\circ u_2+u_2\circ u_1+u_1\star u_1\star u_1=0\in\Hg_{S^3}^4(A,A)$, as desired.

One can continue this construction for any $n\geq1$.
\end{proof}

\subsection{Generalization}\label{dsphereSub}

Fix $d\geq1$. For $\mathbbm{k}$-linear maps $f_1,\ldots,f_m:A\longrightarrow A$, we define $f_1\circ\cdots\circ f_m:A^{\otimes d+1}\longrightarrow A$ in the natural way, where $2\leq m\leq\left\lceil\frac{d+2}{2}\right\rceil$.

Suppose we desire the map $u$ from \eqref{u} to satisfy the property
\begin{equation}\label{SdGen}
\begin{cases}
u(a_1a_2)\cdots u(a_da_{d+1})=u(a_1)u(a_2a_3)\cdots u(a_{d-1}a_d)u(a_{d+1}) & \text{if $d$ is odd,}\\
u(a_1a_2)\cdots u(a_{d-1}a_d)u(a_{d+1})=u(a_1)u(a_2a_3)\cdots u(a_da_{d+1}) & \text{if $d$ is even.}
\end{cases}
\end{equation}

\begin{theorem}\label{SdTheorem}
Fix $d\geq1$. Let $u:A[[t]]\longrightarrow A[[t]]$ be defined as in \eqref{u}.
\begin{enumerate}[(i)]
\item\label{firstthing} If $u$ satisfies \eqref{SdGen} $mod~t^2$, then $u_1\in Z_{S^d}^d(A,A)$.
\item If $u$ satisfies \eqref{SdGen} $mod~t^{n+1}$, then $u$ can be extended so that it satisfies \eqref{SdGen} $mod~t^{n+2}$ if and only if
$$\sum_{m=2}^{\left\lceil\frac{d+2}{2}\right\rceil}\left(\sum_{i_1+\cdots+i_m=n+1}u_{i_1}\circ\cdots\circ u_{i_m}\right)=0\in\Hg_{S^d}^{d+1}(A,A).$$
\end{enumerate}
\end{theorem}
\begin{proof}
Follows from Definition \ref{SdDefn}, the map given in \eqref{SdMap}, and the property in \eqref{SdGen}.
\end{proof}

\begin{example}\label{exdone}
For $d=1$, recall that the higher order Hochschild cohomology recovers the usual Hochschild cohomology. Therefore, one can apply Theorem \ref{SdTheorem} if one desires the map $u$ from \eqref{u} to satisfy the property $u(ab)=u(a)u(b)$.
\end{example}

\begin{example}
The case for $d=2$ recovers precisely what was done in \cite{CS}.
\end{example}

\begin{remark}
Taking $d=3$ in Section \ref{dsphereSub} reduces to what was established in Section \ref{3sphereSub}.
\end{remark}

\begin{corollary}
Fix $d\geq1$. Let $u:A[[t]]\longrightarrow A[[t]]$ be defined as in \eqref{u}. If $u$ satisfies \eqref{SdGen} $mod~t^2$, then the class of $u_1\in\Hg_{S^d}^d(A,A)$ is determined by the isomorphism class of $u$.
\end{corollary}
\begin{proof}
First, we know by Theorem \ref{SdTheorem}\eqref{firstthing} that $u_1$ is a $d$-cocycle. Next we consider two maps: $u(a)=a+u_1(a)t+u_2(a)t^2+\cdots$ and $w(a)=a+w_1(a)t+w_2(a)t^2+\cdots$.
Suppose that we have $f:A[[t]]\longrightarrow A[[t]]$ an isomorphism given by $f(a)=a+f_1(a)t+f_2(a)t^2+\cdots$ such that we desire
\begin{equation}\label{com}
w(f(a))=f(u(a)).
\end{equation}
In other words, the diagram in Figure \ref{figcommute} commutes.
\begin{figure}[ht]
    \centering
$$
\begin{tikzpicture}[scale=2]
\node (a) at (0,0) {$A[[t]]$};
\node (b) at (2,0) {$A[[t]]$};
\node (c) at (0,1) {$A[[t]]$};
\node (d) at (2,1) {$A[[t]]$};
\path[->,font=\small,>=angle 90]
(c) edge node [left] {$f$} (a)
(c) edge node [above] {$u$} (d)
(a) edge node [above] {$w$} (b)
(d) edge node [right] {$f$} (b);
\end{tikzpicture}
$$
    \caption{Commuting diagram}
    \label{figcommute}
\end{figure}
If \eqref{com} is satisfied $mod~t^2$, then we get that $a+f_1(a)t+w_1(a)t=a+u_1(a)t+f_1(a)t$, and hence $(u_1-w_1)(a)=0$.

Using \eqref{evenodd}, we see that when $d$ is odd, we know that $\delta_{d-1}=0$, and when $d$ is even, we know that $\delta_{d-1}=\id$. Regardless, $u_1-w_1\in\Ig(\delta_{d-1})$. This shows that $u_1$ and $w_1$ are in the same class in $\Hg_{S^d}^d(A,A)$. The result follows.
\end{proof}

Notice that all of the equalities contained in \eqref{SdGen} are independent. Observe that $u(ab)=u(a)u(b)$ (see Example \ref{exdone}) clearly implies the others, yet the converse fails. This is mainly because $u(1)$ need not equal 1. The following result generalizes the implications.

\begin{proposition}
Let $u:A[[t]]\longrightarrow A[[t]]$ be defined as in \eqref{u}. If $u$ satisfies \eqref{SdGen} for $d=n$ and for $d=m$, then $u$ satisfies \eqref{SdGen} for $d=n+m$.
\end{proposition}
\begin{proof}
First suppose $n$ is odd and $m$ is even. This, of course, implies that $n+m$ is odd. We assume that $u$ satisfies the following:
\begin{equation}\label{step1}
u(a_1a_2)\cdots u(a_na_{n+1})=u(a_1)u(a_2a_3)\cdots u(a_{n-1}a_n)u(a_{n+1}),
\end{equation}
and
\begin{equation}\label{step2}
u(a_1a_2)\cdots u(a_{m-1}a_m)u(a_{m+1})=u(a_1)u(a_2a_3)\cdots u(a_ma_{m+1}).
\end{equation}
We want to show that $u$ satisfies \eqref{SdGen} for $d=n+m$. We then observe that
\begin{align*}
u&(a_1a_2)\cdots u(a_na_{n+1})u(a_{n+2}a_{n+3})\cdots u(a_{n+m}a_{n+m+1})\\
&=u(a_1)u(a_2a_3)\cdots u(a_{n-1}a_n)u(a_{n+1})u(a_{n+2}a_{n+3})\cdots u(a_{n+m}a_{n+m+1}) \text{~~~by \eqref{step1}}\\
&=u(a_1)u(a_2a_3)\cdots u(a_{n-1}a_n)u(a_{n+1}a_{n+2})\cdots u(a_{n+m-1}a_{n+m})u(a_{n+m+1}) \text{~~~by \eqref{step2}},
\end{align*}
which was what we wanted. The cases for $n$ and $m$ both odd or both even can be done in a similar manner.
\end{proof}

\section{Tertiary Hochschild cohomology}\label{SectionTert}

In this section we impose nontrivial $B$-algebra and $C$-algebra structures on $A$ and establish a deformation theory corresponding to it. This is similar to what was done in \cite{S}.

First we let $\mathcal{Q}=(A,B,C,\varepsilon,\theta)$ be a quintuple. Note that here $A$ is not necessarily commutative. Since $\mathcal{Q}$ is a quintuple, notice that this induces a $B$-algebra structure on $A$ by way of the morphism $\varepsilon:B\longrightarrow A$, and it also induces a $C$-algebra structure on $A$ via the morphism $\varepsilon\circ\theta:C\longrightarrow A$ (see Definition \ref{Quin}).

Next for each $\alpha\in B$ and $x\in C$ we have a map $m_\alpha^x:A\otimes A\longrightarrow A$ given by $m_\alpha^x(a\otimes b)=ab\varepsilon(\alpha)\varepsilon(\theta(x))$, where $a,b\in A$. One can verify that the following are easily satisfied for any $q\in\mathbbm{k}$, $a,b,c\in A$, $\alpha,\beta,\gamma\in B$, and $x,y,z,w\in C$:
\begin{equation}\label{stuff}
\begin{gathered}
m_{\alpha+\beta}^x(a\otimes b)=m_\alpha^x(a\otimes b)+m_\beta^x(a\otimes b),\\
m_\alpha^{x+y}(a\otimes b)=m_\alpha^x(a\otimes b)+m_\alpha^y(a\otimes b),\\
m_{q\alpha}^x(a\otimes b)=qm_\alpha^x(a\otimes b)=m_\alpha^{qx}(a\otimes b),\\
m_{\alpha\beta\theta(z)}^{xy}(a\otimes m_\gamma^w(b\otimes c))=m_{\beta\gamma}^{xyz}(m_\alpha^x(a\otimes b)\otimes c).
\end{gathered}
\end{equation}

Conversely, now we suppose that $\theta:B\longrightarrow C$ is a morphism of commutative $\mathbbm{k}$-algebras, and $A$ is a $\mathbbm{k}$-vector space. Further suppose that we have a family of products $\mathcal{M}_\mathcal{Q}=\{m_\alpha^x:A\otimes A\longrightarrow A~:~\alpha\in B\text{~and~}x\in C\}$ such that $\mathcal{M}_\mathcal{Q}$ satisfies the conditions in \eqref{stuff}. Finally, suppose that $(A,m_1^1)$ is a $\mathbbm{k}$-algebra with the identity element $1_A\in A$.

One can see that $\varepsilon_{\mathcal{M}_\mathcal{Q}}:B\longrightarrow A$ given by $\varepsilon_{\mathcal{M}_\mathcal{Q}}(\alpha)=m_\alpha^1(1_A\otimes1_A)$ and $\varepsilon_{\mathcal{M}_\mathcal{Q}}\circ\theta:C\longrightarrow A$ given by $(\varepsilon_{\mathcal{M}_\mathcal{Q}}\circ\theta)(x)=m_1^x(1_A\otimes1_A)$ are both morphisms of $\mathbbm{k}$-algebras such that $\varepsilon_{\mathcal{M}_\mathcal{Q}}(B)\subseteq\mathcal{Z}(A)$ and $(\varepsilon_{\mathcal{M}_\mathcal{Q}}\circ\theta)(C)\subseteq\mathcal{Z}(A)$, respectively. In particular, both of these maps respect sums, scalars, products, and the identity. As consequence, one has the following:

\begin{proposition}
Consider a morphism $\theta:B\longrightarrow C$ of commutative $\mathbbm{k}$-algebras and a family of products $\mathcal{M}_\mathcal{Q}=\{m_\alpha^x:A\otimes A\longrightarrow A~:~\alpha\in B\text{~and~}x\in C\}$ such that $(A,m_1^1)$ is a $\mathbbm{k}$-algebra with unit $1_A\in A$. Then \eqref{stuff} holds if and only if $\varepsilon_{\mathcal{M}_\mathcal{Q}}:B\longrightarrow A$ given by $\varepsilon_{\mathcal{M}_\mathcal{Q}}(\alpha)=m_\alpha^1(1_A\otimes1_A)$ and $\varepsilon_{\mathcal{M}_\mathcal{Q}}\circ\theta:C\longrightarrow A$ given by $(\varepsilon_{\mathcal{M}_\mathcal{Q}}\circ\theta)(x)=m_1^x(1_A\otimes1_A)$ give a $B$-algebra and $C$-algebra structure on $A$, respectively.
\end{proposition}
\begin{proof}
Follows from the above discussion.
\end{proof}

\subsection{A deformation theory}

Let $\mathcal{Q}=(A,B,C,\varepsilon,\theta)$ be a quintuple. Suppose that for each $i\geq1$ we have a $\mathbbm{k}$-linear map $c_i:A^{\otimes2}\otimes B\otimes C\longrightarrow A$. For each $\alpha\in B$ and $x\in C$, we define $m_{\alpha,t}^x:A[[t]]\otimes A[[t]]\longrightarrow A[[t]]$ determined by
\begin{equation}\label{TertProd}
m_{\alpha,t}^x(a\otimes b)=ab\varepsilon(\alpha\theta(x))+c_1\Big(\begin{pmatrix}a\end{pmatrix}\otimes\begin{pmatrix} x & \alpha\\ & b\\\end{pmatrix}\Big)t+c_2\Big(\begin{pmatrix}a\end{pmatrix}\otimes\begin{pmatrix} x & \alpha\\ & b\\\end{pmatrix}\Big)t^2+\cdots
\end{equation}
where $a,b\in A$. Suppose we desire the family of products $\mathcal{M}_\mathcal{Q}=\{m_{\alpha,t}^x~:~\alpha\in B\text{~and~}x\in C\}$ to satisfy the following associativity condition:
\begin{equation}\label{TertAss}
m_{\alpha\beta\theta(z),t}^{xy}(a\otimes m_{\gamma,t}^w(b\otimes c))=m_{\beta\gamma,t}^{yzw}(m_{\alpha,t}^x(a\otimes b)\otimes c),
\end{equation}
where $a,b,c\in A$, $\alpha,\beta,\gamma\in B$, and $x,y,z,w\in C$.

\begin{remark}\label{UST2}
Taking $C=\mathbbm{k}$, one recovers the family of products $\mathcal{M}$ discussed in \cite{S}, whereas taking $C=B=\mathbbm{k}$ reduces to the usual product $m_t$ studied in \cite{Ger}.
\end{remark}

For $\mathbbm{k}$-linear maps $f,g:A^{\otimes2}\otimes B\otimes C\longrightarrow A$, we define $f\circ g:A^{\otimes3}\otimes B^{\otimes3}\otimes C^{\otimes4}\longrightarrow A$ by the following:
$$
(f\circ g)\Big(\begin{pmatrix}a\\\end{pmatrix}\otimes\begin{pmatrix} x & \alpha\\ & b\\\end{pmatrix}\otimes\begin{pmatrix} y & z & \beta\\ & w & \gamma\\ & & c\end{pmatrix}\Big)
$$
$$
=f\Big(\begin{pmatrix}g\Big(\begin{pmatrix}a\end{pmatrix}\otimes\begin{pmatrix} x & \alpha\\ & b\\\end{pmatrix}\Big)\end{pmatrix}\otimes\begin{pmatrix} yzw & \beta\gamma\\ & c\\\end{pmatrix}\Big)-f\Big(\begin{pmatrix}a\end{pmatrix}\otimes\begin{pmatrix} xy & \alpha\beta\theta(z)\\ & g\Big(\begin{pmatrix}b\end{pmatrix}\otimes\begin{pmatrix} w & \gamma\\ & c\\\end{pmatrix}\Big)\end{pmatrix}\Big).
$$

\begin{theorem}\label{TertBig}
Let $\mathcal{Q}=(A,B,C,\varepsilon,\theta)$ be a quintuple. Suppose $\mathcal{M}_\mathcal{Q}=\{m_{\alpha,t}^x~:~\alpha\in B\text{~and~}x\in C\}$ is the family of products defined as in \eqref{TertProd}.
\begin{enumerate}[(i)]
\item\label{TertFirst} If the family of products $\mathcal{M}_\mathcal{Q}$ satisfies \eqref{TertAss} $mod~t^2$, then $c_1\in Z^2((A,B,C,\varepsilon,\theta);A)$.
\item If the family of products $\mathcal{M}_\mathcal{Q}$ satisfies \eqref{TertAss} $mod~t^{n+1}$, then $\mathcal{M}_\mathcal{Q}$ can be extended so that the family of products satisfies \eqref{TertAss} $mod~t^{n+2}$ if and only if
$$
\sum_{i+j=n+1}c_i\circ c_j=0\in\Hg^3((A,B,C,\varepsilon,\theta);A).
$$
\end{enumerate}
\end{theorem}
\begin{proof}
For $(i)$, in order to satisfy \eqref{TertAss} $mod~t^2$ we need
$$
c_1\Big(\begin{pmatrix}a\end{pmatrix}\otimes\begin{pmatrix} xy & \alpha\beta\theta(z)\\ & bc\varepsilon(\gamma\theta(w))\\\end{pmatrix}\Big)+a\varepsilon(\alpha\beta\theta(xyz))c_1\Big(\begin{pmatrix}b\end{pmatrix}\otimes\begin{pmatrix} w & \gamma\\ & c\\\end{pmatrix}\Big)
$$
$$
=c_1\Big(\begin{pmatrix}ab\varepsilon(\alpha\theta(x))\end{pmatrix}\otimes\begin{pmatrix} yzw & \beta\gamma\\ & c\\\end{pmatrix}\Big)+c_1\Big(\begin{pmatrix}a\end{pmatrix}\otimes\begin{pmatrix} x & \alpha\\ & b\\\end{pmatrix}\Big)c\varepsilon(\beta\gamma\theta(yzw))
$$
which can be rearranged as
$$
a\varepsilon(\alpha\beta\theta(xyz))c_1\Big(\begin{pmatrix} b\\\end{pmatrix}\otimes\begin{pmatrix} w & \gamma\\ & c\end{pmatrix}\Big)-c_1\Big(\begin{pmatrix} ab\varepsilon(\alpha\theta(x))\\\end{pmatrix}\otimes\begin{pmatrix} yzw & \beta\gamma\\ & c\end{pmatrix}\Big)
$$
$$
+c_1\Big(\begin{pmatrix} a\\\end{pmatrix}\otimes\begin{pmatrix} xy & \alpha\beta\theta(z)\\ & bc\varepsilon(\gamma\theta(w))\end{pmatrix}\Big)-c_1\Big(\begin{pmatrix} a\\\end{pmatrix}\otimes\begin{pmatrix} x & \alpha\\ & b\end{pmatrix}\Big)c\varepsilon(\beta\gamma\theta(yzw))=0.
$$
Hence $\gimel^2(c_1)=0$, and so $c_1\in\Ker(\gimel^2)$. This implies that $c_1$ is a 2-cocycle, and thus $c_1\in Z^2((A,B,C,\varepsilon,\theta);A)$.

For $(ii)$, we consider the case $n=1$. Assuming the family of products $\mathcal{M}_\mathcal{Q}$ satisfies \eqref{TertAss} $mod~t^2$, if we further desire $\mathcal{M}_\mathcal{Q}$ to satisfy \eqref{TertAss} $mod~t^3$, we then get that
\begin{equation}\label{TertExtend}
\begin{gathered}
c_2\Big(\begin{pmatrix}a\end{pmatrix}\otimes\begin{pmatrix} xy & \alpha\beta\theta(z)\\ & bc\varepsilon(\gamma\theta(w))\\\end{pmatrix}\Big)+c_1\Big(\begin{pmatrix}a\end{pmatrix}\otimes\begin{pmatrix} xy & \alpha\beta\theta(z)\\ & c_1\Big(\begin{pmatrix}b\end{pmatrix}\otimes\begin{pmatrix} w & \gamma\\ & c\\\end{pmatrix}\Big)\end{pmatrix}\Big)\\
+a\varepsilon(\alpha\beta\theta(xyz))c_2\Big(\begin{pmatrix}b\end{pmatrix}\otimes\begin{pmatrix} w & \gamma\\ & c\\\end{pmatrix}\Big)=c_2\Big(\begin{pmatrix}ab\varepsilon(\alpha\theta(x))\end{pmatrix}\otimes\begin{pmatrix} yzw & \beta\gamma\\ & c\\\end{pmatrix}\Big)\\
+c_1\Big(\begin{pmatrix}c_1\Big(\begin{pmatrix}a\end{pmatrix}\otimes\begin{pmatrix} x & \alpha\\ & b\\\end{pmatrix}\Big)\end{pmatrix}\otimes\begin{pmatrix} yzw & \beta\gamma\\ & c\\\end{pmatrix}\Big)+c_2\Big(\begin{pmatrix}a\end{pmatrix}\otimes\begin{pmatrix} x & \alpha\\ & b\\\end{pmatrix}\Big)c\varepsilon(\beta\gamma\theta(yzw)).
\end{gathered}
\end{equation}
Notice that \eqref{TertExtend} can be rewritten as $\gimel^2(c_2)=c_1\circ c_1$, which describes the obstruction. Hence since $c_1\circ c_1\in\Ig(\gimel^2)$, then $c_1\circ c_1=0\in\Hg^3((A,B,C,\varepsilon,\theta);A)$. Observe that one can do this for any $n\geq1$ in order to extend associativity of the family of products $\mathcal{M}_\mathcal{Q}$, as is traditional for these types of deformation arguments.
\end{proof}

\begin{remark}
As one would expect based on Example \ref{UST} and Remark \ref{UST2}, we see that Theorem \ref{TertBig} reduces to the known deformation theory results corresponding to the secondary and usual Hochschild cohomologies found in \cite{S} and \cite{Ger} when one takes $C=\mathbbm{k}$ and $C=B=\mathbbm{k}$, respectively.
\end{remark}

\begin{corollary}
Let $\mathcal{Q}=(A,B,C,\varepsilon,\theta)$ be a quintuple. Suppose $\mathcal{M}_\mathcal{Q}=\{m_{\alpha,t}^x~:~\alpha\in B\text{~and~}x\in C\}$ is the family of products defined as in \eqref{TertProd}. If the family of products $\mathcal{M}_\mathcal{Q}$ satisfies \eqref{TertAss} $mod~t^2$, then the class of $c_1\in\Hg^2((A,B,C,\varepsilon,\theta);A)$ is determined by the isomorphism class of $\mathcal{M}_\mathcal{Q}$.
\end{corollary}
\begin{proof}
First, we know by Theorem \ref{TertBig}\eqref{TertFirst} that $c_1$ is a $2$-cocycle. Next we consider two families of products $\{m_{\alpha,t}^x~:~\alpha\in B\text{~and~}x\in C\}$ and $\{p_{\alpha,t}^x~:~\alpha\in B\text{~and~}x\in C\}$:
$$
m_{\alpha,t}^x(a\otimes b)=ab\varepsilon(\alpha\theta(x))+c_1\Big(\begin{pmatrix}a\end{pmatrix}\otimes\begin{pmatrix} x & \alpha\\ & b\\\end{pmatrix}\Big)t+c_2\Big(\begin{pmatrix}a\end{pmatrix}\otimes\begin{pmatrix} x & \alpha\\ & b\\\end{pmatrix}\Big)t^2+\cdots
$$
and
$$
p_{\alpha,t}^x(a\otimes b)=ab\varepsilon(\alpha\theta(x))+d_1\Big(\begin{pmatrix}a\end{pmatrix}\otimes\begin{pmatrix} x & \alpha\\ & b\\\end{pmatrix}\Big)t+d_2\Big(\begin{pmatrix}a\end{pmatrix}\otimes\begin{pmatrix} x & \alpha\\ & b\\\end{pmatrix}\Big)t^2+\cdots.
$$
Suppose $f:A[[t]]\longrightarrow A[[t]]$ is an isomorphism given by $f(a)=a+f_1(a)t+f_2(a)t^2+\cdots$ such that we desire
\begin{equation}\label{com2}
p_{\alpha,t}^x(f(a)\otimes f(b))=f(m_{\alpha,t}^x(a\otimes b)).
\end{equation}
Equivalently, the diagram in Figure \ref{figcommute2} commutes.
\begin{figure}[ht]
    \centering
$$
\begin{tikzpicture}[scale=2]
\node (a) at (0,0) {$A[[t]]\otimes A[[t]]$};
\node (b) at (2,0) {$A[[t]]$};
\node (c) at (0,1) {$A[[t]]\otimes A[[t]]$};
\node (d) at (2,1) {$A[[t]]$};
\path[->,font=\small,>=angle 90]
(c) edge node [left] {$f\otimes f$} (a)
(c) edge node [above] {$m_{\alpha,t}^x$} (d)
(a) edge node [above] {$p_{\alpha,t}^x$} (b)
(d) edge node [right] {$f$} (b);
\end{tikzpicture}
$$
    \caption{Commuting diagram}
    \label{figcommute2}
\end{figure}
If \eqref{com2} is satisfied $mod~t^2$, then we get
\begin{equation}\label{TertClass}
\begin{gathered}
ab\varepsilon(\alpha\theta(x))+d_1\Big(\begin{pmatrix}a\end{pmatrix}\otimes\begin{pmatrix} x & \alpha\\ & b\\\end{pmatrix}\Big)t+a\varepsilon(\alpha\theta(x))f_1(b)t+f_1(a)b\varepsilon(\alpha\theta(x))t\\
=ab\varepsilon(\alpha\theta(x))+f_1(ab\varepsilon(\alpha\theta(x)))t+c_1\Big(\begin{pmatrix}a\end{pmatrix}\otimes\begin{pmatrix} x & \alpha\\ & b\\\end{pmatrix}\Big)t.
\end{gathered}
\end{equation}
Rearranging \eqref{TertClass} yields
\begin{equation}\label{TertClass2}
\begin{gathered}
a\varepsilon(\alpha\theta(x))f_1(b)-f_1(ab\varepsilon(\alpha\theta(x)))+f_1(a)b\varepsilon(\alpha\theta(x))\\
=c_1\Big(\begin{pmatrix}a\end{pmatrix}\otimes\begin{pmatrix} x & \alpha\\ & b\\\end{pmatrix}\Big)-d_1\Big(\begin{pmatrix}a\end{pmatrix}\otimes\begin{pmatrix} x & \alpha\\ & b\\\end{pmatrix}\Big).
\end{gathered}
\end{equation}
One can then rewrite \eqref{TertClass2} as $c_1-d_1=\gimel^1(f_1)$. This shows that $c_1-d_1\in\Ig(\gimel^1)$ and hence $c_1$ and $d_1$ are in the same class in $\Hg^2((A,B,C,\varepsilon,\theta);A)$.
\end{proof}

\subsection{Extensions and remarks}

One observes that the tertiary Hochschild cohomology controls a deformation theory on $A[[t]]$ that has both nontrivial $B$-algebra and $C$-algebra structures. However, as mentioned in \cite{CL}, there is nothing special about the tertiary Hochschild cohomology. As one could imagine, one can extend to a so-called quaternary Hochschild cohomology to investigate deformations of $A[[t]]$ that have three additional algebra structures (coming from $B$, $C$, and $D$, say). In short, however finitely many distinct nontrivial algebra structures that one desires to impose on $A[[t]]$, one can conceivably devise the appropriate generalized Hochschild cohomology that will control that deformation theory. 


\end{document}